\documentclass[12pt]{amsart}
\usepackage{amsmath,amssymb,amsthm,amsfonts}
\usepackage{setspace}
\usepackage{moreverb}
\usepackage[dvips]{graphicx}

\usepackage[latin1]{inputenc}
\usepackage[T1]{fontenc}

\usepackage{color}

\definecolor{r}{rgb}{1,0,0} 
\definecolor{b}{rgb}{0,0,1} 
\definecolor{g}{rgb}{0,1,0} 

\newcommand{\ignore}[1]{}

\newtheorem{theorem}{Theorem}[section]
\newtheorem{lemma}[theorem]{Lemma}
\newtheorem{corollary}[theorem]{Corollary}
\newtheorem{proposition}[theorem]{Proposition}

\theoremstyle{definition}
\newtheorem{definition}[theorem]{Definition}
\newtheorem{example}[theorem]{Example}

\theoremstyle{remark}
\newtheorem{remark}[theorem]{Remark}

\numberwithin{equation}{section}

\input xy
\xyoption{all}

\newcommand{\gr}{\mbox{gr}\,}

\newcommand{\bD}{\mathbf{D}}

\newcommand{\bN}{{\mathbb{N}}}
\newcommand{\bZ}{{\mathbb{Z}}}
\newcommand{\bQ}{{\mathbb{Q}}}
\newcommand{\bR}{\mathbb{R}}

\newcommand{\bx}{\mathbf{x}}

\newcommand{\dm}{\mathrm {dim }}

\newcommand{\depth}{\mathrm {depth}}

\newcommand{\Supp}{\mbox{\rm{Supp}} }
\newcommand{\supp}{\mbox{\rm{supp}} }

\newcommand{\im}{\mathrm {Im\, }}

\newcommand{\length}{\mathrm{length}}

\newcommand{\Ext}{\mbox{\rm{Ext}} }

\newcommand{\fM}{{\mathfrak m}}

\newcommand{\fp}{\mathfrak{p}}

\newcommand{\cF}{{\mathcal F}}

\newcommand{\la}{{\lambda}}

\newcommand{\lra}{{\longrightarrow}}

\newcommand{\op}{{\oplus}}

\begin{document}

\title[Lyubeznik numbers and linear strands]{Lyubeznik numbers of local rings and linear strands of graded ideals}

\author[J. \`Alvarez Montaner]{Josep \`Alvarez Montaner}
\address{Departament de Matem\`atiques\\
Universitat Polit\`ecnica de Catalunya\\ Av. Diagonal 647,
Barcelona 08028, SPAIN} \email{Josep.Alvarez@upc.edu}

\author[K. Yanagawa]{Kohji Yanagawa}
\address{Department of Mathematics, Kansai University, Suita 564-8680, Japan}
\email{yanagawa@ipcku.kansai-u.ac.jp}

\thanks{The first author was partially 
supported by Generalitat de Catalunya 2014SGR-634 project
and Spanish Ministerio de Econom\'ia y Competitividad
MTM2015-69135-P. The
second author was partially supported by JSPS KAKENHI 25400057.}



\newcommand{\ba}{\mathbf{a}}
\newcommand{\bb}{\mathbf{b}}
\newcommand{\bc}{\mathbf{c}}
\newcommand{\be}{\mathbf{e}}
\newcommand{\gmod}{\operatorname{* mod}}
\newcommand{\Sq}{\operatorname{Sq}}
\newcommand{\relint}{\operatorname{rel-int}}
\newcommand{\kk}{\Bbbk}
\newcommand{\KK}{\mathbb{K}}
\newcommand\const{\underline{\kk}}
\newcommand\Dcom{\mathcal{D}^\bullet}
\newcommand\cExt{{\mathcal Ext}}
\newcommand\Db{{\mathsf D}^b}
\newcommand\GG{\mathbb G}
\newcommand{\bA}{\mathbf{A}}
\newcommand{\EE}{\mathbb{E}}
\newcommand\LL{\mathbb L}

\begin{abstract}

In this work we introduce a new set of invariants associated to the
linear strands of a minimal free resolution of a $\bZ$-graded ideal
$I\subseteq R=\kk[x_1, \ldots, x_n]$. We also prove that these
invariants satisfy some properties analogous to those of Lyubeznik
numbers of local rings. In particular, they satisfy a
consecutiveness property that we prove first for the Lyubeznik
table. For the case of squarefree monomial ideals we get more
insight on the relation between Lyubeznik numbers and the linear
strands of their associated Alexander dual ideals. Finally, we prove
that Lyubeznik numbers of Stanley-Reisner rings are not only an
algebraic invariant but also a topological invariant, meaning that
they depend on the homeomorphic class of the geometric realization
of the associated simplicial complex and the characteristic of the
base field.

\end{abstract}

\maketitle

\section{Introduction}

Let $A$ be a noetherian local ring that admits a surjection from
an $n$-dimensional regular local ring $(R,\mathfrak{m})$ containing its residue field $\kk$, and $I\subseteq R$ be the kernel
of the surjection. In \cite{Ly93}, Lyubeznik introduced a new set of invariants $\lambda_{p,i} (A)$ as the $p$-th Bass number
of the local cohomology module $H^{n-i}_I(R)$, that is
$$\lambda_{p,i} (A):= \mu^p(\fM, H_{I}^{n-i}(R))=\dim_{\kk} \Ext_R^{p}(\kk, H_I^{n-i}(R))$$ and they depend only on $A$, $i$ and $p$, but not on the choice
of $R$ or the surjection $R\lra A$. In the seminal works of
Huneke-Sharp \cite{HS93} and Lyubeznik \cite{Ly93} it is proven that
these Bass numbers are all finite.  Denote $d=\dm A$, Lyubeznik
numbers satisfy the following properties\footnote{Property iii) was
shown to us by R.~Garc\'ia-L\'opez (see \cite{Alv13} for details).}:
\begin{itemize}
\item[i)]     $\la_{p,i}(A) \ne 0$ implies  $0 \le p \le i \le d$.
\item[ii)]   $\la_{d,d}(A)\neq 0$.
\item[iii)]   {\it Euler characteristic:} $$\sum_{0\leq p,i \leq d} (-1)^{p-i} \la_{p,i}(A)=1.$$
\end{itemize}

Therefore, we can collect them in
the so-called {\it Lyubeznik table}:
$$\Lambda(A)  = \left(
                    \begin{array}{ccc}
                      \la_{0,0} & \cdots & \la_{0,d}  \\
                       & \ddots & \vdots \\
                       &  & \la_{d,d} \\
                    \end{array}
                  \right)
$$ and we say that the Lyubeznik table is {\it trivial} if $ \la_{d,d}=1$ and the rest of these invariants vanish.

\vskip 2mm

Despite its algebraic nature, Lyubeznik numbers also provide some geometrical and topological information
as it was already pointed out in \cite{Ly93}. For instance, in the case of isolated singularities, Lyubeznik numbers can
be described in terms of certain singular cohomology groups in characteristic zero (see \cite{GS98}) or
\'etale cohomology groups in positive characteristic (see \cite{BlBo}, \cite{Bl3}).
The highest Lyubeznik number $\la_{d,d}(A)$ can be described using the so-called Hochster-Huneke
graph as it has been proved in \cite{Lyu06}, \cite{Zha07}. However very little is known about
the possible configurations of Lyubeznik tables except for low dimension cases \cite{Kaw00},
\cite{Wa2} or the just mentioned case of isolated singularities.

\vskip 2mm

In Section 2 we will give some new constraints to the possible
configurations of Lyubeznik tables. Namely, the main result,
Theorem~\ref{consecutiveness},
establishes some consecutiveness of  the non-vanishing  superdiagonals of the Lyubeznik tables
using spectral sequence arguments.

\vskip 2mm

In Section 3 we introduce a new set of invariants associated to the linear strands of a minimal free resolution
of a $\bZ$-graded ideal $I\subseteq R=\kk[x_1, \ldots, x_n]$.  It turns out that these new
invariants satisfy some analogous properties to those of Lyubeznik numbers including the aforementioned consecutiveness
property. Moreover, we provide a Thom-Sebastiani
type formula for these invariants that is a refinement of the formula for Betti numbers given by Jacques and Katzman in \cite{JK05}.
This section should be of independent interest and we hope it could be further developed in future work.

\vskip 2mm

In the rest of the paper we treat the case where $I$ is a monomial ideal in a polynomial ring $R=\kk[x_1, \ldots, x_n]$, and $\fM=(x_1, \ldots, x_n)$ is the graded maximal ideal.
Bass numbers are invariant with respect to completion so we consider $\lambda_{p,i} (R/I)=\lambda_{p,i} (\widehat{R}/I\widehat{R})$
where $\widehat{R}=\kk[\![x_1, \ldots, x_n]\!]$. In this sense, our study on the Lyubeznik tables of monomial ideals is a (very) special case of that for local rings.
However, advanced techniques in combinatorial commutative algebra are very effective in this setting,
and we can go much further than general case, so we hope that monomial ideals are good ``test cases'' for the study of Lyubeznik tables.

\vskip 2mm

Since local cohomology modules satisfy $H_I^i(R) \cong H_{\sqrt{I}}^i(R)$
we often will assume that a monomial ideal $I$ is {\it squarefree}, i.e., $I=\sqrt{I}$.
In this case, $I$ coincides with the {\it Stanley-Reisner ideal} $I_\Delta$ of a simplicial complex $\Delta \subseteq 2^{\{1, \ldots, n\}}$, more precisely,
$$I=I_\Delta := (\prod_{i \in F} x_i \mid F \subseteq \{1, \ldots, n\}, F \not \in \Delta ).$$
The Stanley-Reisner ring $R/I_\Delta$ is one of the most fundamental tools in combinatorial commutative algebra,
and it is known that $R/I_\Delta$ reflects topological properties of the geometric realization $|\Delta|$ of $\Delta$ in several ways.

\vskip 2mm

In Section 4 we get a deeper insight to the relation, given by the
first author and A.~Vahidi in \cite{AV11}, between Lyubeznik numbers
of monomial ideals and the linear strands of their associated
Alexander dual ideals. In particular, we give a different approach
to the fact proved in \cite{Alv13} that if $R/I_\Delta$ is
sequentially Cohen-Macaulay, then its Lyubeznik table is trivial. We
also provide a Thom-Sebastiani type formula for Lyubeznik numbers.

\vskip 2mm

One of the main results of this paper is left for Section 5. Namely,  Theorem~\ref{topological}  states that Lyubeznik numbers of Stanley-Reisner rings are not only an algebraic invariant
but also a topological invariant, meaning that the Lyubeznik numbers of $R/I_\Delta$ depend on the homeomorphic class of the geometric realization
$|\Delta|$ of $\Delta$ and the characteristic of the base field.



\vskip 2mm

The proof of this result is quite technical and irrelevant to the other parts of the paper, so we decided to put it in the final section.
We also remark that this result holds in a wider setting. More precisely, if $R$ is a normal simplicial semigroup ring which is Gorenstein, and
$I$ is a monomial ideal, then the corresponding result holds.  We will work  in this general setting, since the proof is the same as in the
polynomial ring case.

\section{Consecutiveness of nontrivial superdiagonals of the Lyubeznik table}

To give a full description of the possible configurations of
Lyubeznik tables of any local ring seems to be a very difficult task
and only a few results can be found in the literature. The aim of
this section is to give some constraints to the possible
configurations of Lyubeznik tables, aside from the Euler
characteristic formula.

\vskip 2mm

Let $(R,\mathfrak{m})$ be a regular local ring  of dimension $n$ containing its residue field $\kk$, and $I\subseteq R$
be any ideal with $\dim R/I=d$. For each $j \in \bN$ with $0 \le j \le d$, set $$\rho_j(R/I) =\sum_{i=0}^{d-j} \lambda_{i, i+j}(R/I).$$
For example, $\rho_0(R/I)$ (resp. $\rho_1(R/I)$) is the sum of the entries in the diagonal (resp. superdiagonal) of the Lyubeznik table $\Lambda(R/I)$.
Clearly,  $\sum_{j \in \bN}(-1)^j\rho_j(R/I)=1$. We say $\rho_j(R/I)$ is {\it non-trivial}, if
$$
\rho_j(R/I) \ge
\begin{cases}
2 & \text{if $j=0$,} \\
1 & \text{if $j \ge 1$.}
\end{cases}
$$
Clearly, $\Lambda(R/I)$ is non-trivial if and only if $\rho_j(R/I)$ is non-trivial for some $j$.

It is easy to see that   $\lambda_{0,d}(R/I)=0$ if $d \ge 1$ and  $\lambda_{0,d}(R/I)=1$ if $d =0$, that is, $\rho_d(R/I)$ is always trivial.

\vskip 2mm

A key fact that we are going to use in this section is that local
cohomology modules have a natural structure over the ring of
$k$-linear differential operators $D_{R|k}$ (see \cite{Ly93},
\cite{Ly97}). In fact they are  $D_{R|k}$-modules of finite length
(see \cite[Thm. 2.7.13]{Bj79} and \cite[Ex.2.2]{Ly93} for the case
of characteristic zero and \cite[Thm. 5.7]{Ly97} in positive
characteristic). In particular, Lyubeznik numbers are nothing but
the length as a $D_{R|k}$-module of the local cohomology modules
$H_{\mathfrak{m}}^p (H_I^{n-i}(R))$, i.e. $$\la_{p,i}(R/I)=
\length_{D_{R|k}}(H_{\mathfrak{m}}^p (H_I^{n-i}(R))).$$ The
$D_{R|k}$-module length, that will be denoted simply as $e(-)$, is
an additive function, i.e given a short exact sequence of holonomic
$D_{R|k}$-modules $0\lra M_1 \lra M_2 \lra M_3\lra 0$ we have
$$e(M_2)=e(M_1)+e(M_3).$$

\vskip 2mm

The main result of this section is the following:

\begin{theorem}\label{consecutiveness}
Let $(R,\mathfrak{m})$ be a regular local ring  of dimension $n$ containing its residue field $\kk$, and $I\subseteq R$
be any ideal with $\dim R/I=d$. Then:

\begin{itemize}
\item[$\bullet$] If $\rho_j(R/I)$ is non-trivial for some $j$ with $0 < j < d$,
then either $\rho_{j-1}(R/I)$ or  $\rho_{j+1}(R/I)$ is non-trivial.

\vskip 2mm

\item[$\bullet$] If   $\rho_0(R/I)$ is non-trivial, then so is
 $\rho_1(R/I)$.
\end{itemize}
\end{theorem}


\begin{proof}
Consider Grothendieck's spectral sequence $$E_2^{p,n-i}= H_{\fM}^p(H_I^{n-i}(R))\Longrightarrow H_{\fM}^{p+n-i}(R).$$
This is a spectral sequence of $D_{R|k}$-modules where $\lambda_{p,i}(R/I)=e(E_2^{p,n-i})$. Notice also that the local
cohomology modules $H_{\fM}^{r}(R)$ vanish for all $r\neq n$ and in this case $e(H_{\fM}^{n}(R))=1$.

\vskip 2mm

We will prove the assertion by contradiction. So assume that $\rho_j(R/I)$ is non-trivial for some $0 < j < d$, but both
$\rho_{j-1}(R/I)$ and  $\rho_{j+1}(R/I)$ are trivial (the case $j=0$ can be proved by a similar argument).
We have some $p,i$ with $i=p+j$ such that $\lambda_{p,i}(R/I) \ne 0$ (equivalently, ${E}_2^{p,n-i} \ne 0$).
Consider the maps on $E_2$-terms
$${E}_2^{p-2,n-i+1} \stackrel{d_2}{\longrightarrow} {E}_2^{p,n-i} \stackrel{d'_2}{\longrightarrow} {E}_2^{p+2,n-i-1}.$$
We will show  that  $d_2=d'_2=0$.

\vskip 2mm

Consider first the case $j>1$.  We have ${E}_2^{p-2,n-i+1}= {E}_2^{p+2,n-i-1}=0$  just because 
$e({E}_2^{p-2,n-i+1})= \lambda_{p-2,i-1}(R/I)$ and $ e({E}_2^{p+2,n-i-1})= \lambda_{p+2,i+1}(R/I)$ 
concern $\rho_{j+1}(R/I)$ and $\rho_{j-1}(R/I)$ respectively. Therefore $d_2=d'_2=0$ is satisfied trivially.
When $j=1$, i.e. the case when $(p+2,n-i-1) = (d,n-d)$, we have:
$${E}_2^{d-4,n-d+2} \stackrel{d_2}{\longrightarrow} {E}_2^{d-2,n-d+1} \stackrel{d'_2}{\longrightarrow} {E}_2^{d,n-d}.$$
The triviality of $\rho_2(R/I)$ and $\rho_0(R/I)$
means that ${E}_2^{d-4,n-d+2}=0$ and $\lambda_{d,d}= e({E}_2^{d,n-d})=1$ so $d_2=0$. Now 
we assume that the map $d'_2:  {E}_2^{d-2,n-d+1} \to  {E}_2^{d,n-d}$ is non-zero. 
Then $ \im d'_2 = {E}_2^{d,n-d}$
due to the fact that ${E}_2^{d,n-d}$ is a simple $D_{R|k}$-module. It follows
that ${E}_3^{d,n-d}= {E}_2^{d,n-d}/ \im d'_2=0$ so  
$$0={E}_3^{d,n-d} ={E}_4^{d,n-d} = \cdots = {E}_\infty^{d,n-d}.$$
On the other hand, since $\rho_0(R/I)$ is trivial, we have  $$0={E}_2^{i,n-i} ={E}_3^{i,n-i}= \cdots = {E}_\infty^{i,n-i}$$ 
for all $i < d$. Therefore we get a contradiction since, by the general theory of spectral sequences, there exists a filtration
\begin{equation} \label{filtration}
 0\subseteq \cF_n^n \subseteq \cdots \subseteq \cF_1^n \subseteq H_{\fM}^{n}(R),
 \end{equation}
 where ${E}_\infty^{i,n-i}=\cF_i^n/\cF_{i+1}^n$.
 
 \vskip 2mm
 
Anyway, we have shown that $d_2=d_2'=0$ in all cases, and this implies that ${E}_3^{p,n-i}= {E}_2^{p,n-i} \ne 0$.
Now we consider the maps on $E_3$-terms
$${E}_3^{p-3,n-i+2} \stackrel{d_3}{\longrightarrow} {E}_3^{p,n-i} \stackrel{d'_3}{\longrightarrow} {E}_3^{p+3,n-i-2}.$$
Since ${E}_3^{p-3,n-i+2}$ and ${E}_3^{p+3,n-i-2}$ concern $\rho_{j+1}(R/I)$ and $\rho_{j-1}(R/I)$ respectively, we have
 $d_3=d_3'=0$ by the same argument as above. Hence we have ${E}_4^{p,n-i}= {E}_3^{p,n-i} \ne 0$.
Repeating this argument, we have $0 \ne {E}_2^{p,n-i} = {E}_3^{p,n-i}=\cdots = {E}_\infty^{p,n-i}$ so we get a contradiction
with the fact that $H_{\fM}^{p+n-i}(R)=0$ (recall that $j=i-p \ne 0$). 
\end{proof}


\vskip 2mm

The behavior of the consecutive superdiagonals is reflected in the following example. 

\begin{example}
Let $I \subseteq R=\kk[\![x_1, \ldots, x_8]\!]$ be the Alexander dual ideal of the edge ideal of an $8$-cycle, i.e. 
$I^{\vee}=(x_1x_2,x_2x_3,\dots, x_7x_8, x_8x_1).$ Using the results of \cite{AV11} we get the Lyubeznik table 
$$\Lambda(R/I)=\begin{pmatrix}
  0 & 0 & 0 & 0 & 1 & 0 & 0 \\
   & 0 & 0& 0 & 0 & 0 & 0\\
   &  & 0 & 0 & 0& 1 & 0 \\
   &  &  & 0 & 0 & 1 & 0\\
    &  &  &  & 0 & 0 & 0 \\
   &  &  &  &  & 0 & 1 \\
   &  & & & &  & 1
\end{pmatrix}.$$
Notice that $\rho_0(R/I)$ being trivial does not imply that $\rho_1(R/I)=0$.

\end{example}

\begin{remark}
Using similar spectral sequence  arguments to those considered in Theorem~\ref{consecutiveness},
Kawasaki \cite{Kaw00} and Walther \cite{Wa2} described the possible Lyubeznik tables for rings up to dimension two. 
Namely, their result is:

\begin{itemize}

\item If  $d = 2$, then $\lambda_{2,2}(R/I)-1 =\lambda_{0,1}(R/I)$ and the other Lyubeznik numbers are 0.

\end{itemize}

If we take a careful look at the spectral sequence we can also obtain the following:

\begin{itemize}
\item If  $d \ge 3$, then $\lambda_{2,d}(R/I) =\lambda_{0,d-1}(R/I)$ and
\begin{eqnarray*}
\lambda_{1,d-1}(R/I) \le \lambda_{3,d}(R/I) &\le& \lambda_{1,d-1}(R/I) +  \lambda_{0,d-2}(R/I) \\
&\le& \lambda_{3,d}(R/I)+ \lambda_{2,d-1}(R/I).
\end{eqnarray*}
For $d = 3$  we can refine the last inequality, that is,
$$\lambda_{1,2}(R/I) +  \lambda_{0,1}(R/I) = \lambda_{3,3}(R/I)+ \lambda_{2,2}(R/I)-1.$$

\end{itemize}

\vskip 2mm

Indeed, using the filtration \eqref{filtration} we have 
$$\sum_{i=0}^d e(E_{\infty}^{i,n-r-i})=e(H_{\fM}^{n-r}(R))= \begin{cases}
1 & \text{if $r=0$,}\\
0 & \text{otherwise.}
\end{cases}$$ Then the result follows considering the 
 differentials $d_2: E_2^{0,n-d+1} \longrightarrow E_2^{2,n-d}$, $d_2:E_2^{1,n-d+1} \longrightarrow E_2^{3,n-d}$,  
$d_2:E_2^{0,n-d+2} \longrightarrow E_2^{2,n-d+1}$ and 
$d_3:E_3^{0,n-d+2} \longrightarrow E_3^{3,n-d}$. Finally, we point out that
$E_3^{0,n-d+1}=E_{\infty}^{0,n-d+1}$, $E_3^{1,n-d+1}=E_{\infty}^{1,n-d+1}$, $E_3^{2,n-d}=E_{\infty}^{2,n-d}$, $E_4^{0,n-d+2}=E_{\infty}^{0,n-d+2}$ and $E_4^{3,n-d}=E_{\infty}^{3,n-d}$. 
\end{remark}

\section{Linear strands of minimal free resolutions of $\bZ$-graded ideals}
Throughout this section we will consider $\bZ$-graded ideals $I$ in
the polynomial ring $R=\kk[x_1, \ldots, x_n]$, in particular $I$ is
not necessarily a monomial ideal.
For simplicity, we will assume that $I \ne 0$. 
The minimal $\bZ$-graded free resolution of $I$ is an exact sequence of free $\bZ$-graded
modules:
\begin{equation}\label{resolution of I}
L_{\bullet}(I): \hskip 3mm \xymatrix{ 0 \ar[r]& L_{n}
\ar[r]^{d_{n}}& \cdots \ar[r]& L_1 \ar[r]^{d_1}& L_{0} \ar[r]& I
\ar[r]& 0},
\end{equation}
where the $i$-th term is of the form
$$L_i =\bigoplus_{j \in \bZ} R(-j)^{\beta_{i,j}(I)},$$ and the matrices of the
morphisms $d_i: L_i \longrightarrow L_{i-1}$ do not contain
invertible elements.  The {\it Betti numbers} of $I$ are the
invariants $\beta_{i,j}(I)$. Notice that $L_i \cong R^{\beta_i(I)}$ as underlying $R$-modules
where, for each $i$, we set $\beta_i (I) :=\sum_{j \in \bZ} \beta_{i,j}(I)$. Hence,
 \eqref{resolution of I} implies that
$$\sum_{0 \le i \le n}(-1)^i\beta_i(I)= \operatorname{rank}_R (I)=1.$$

\vskip 2mm

Given $r \in \bN$, we also consider the
{\it $r$-linear strand} of $L_{\bullet}(I)$:
$$\mathbb{L}_{\bullet}^{<r>}(I): \hskip 3mm \xymatrix{ 0 \ar[r]&
L_{n}^{<r>} \ar[r]^{d_{n}^{<r>}}& \cdots \ar[r]& L_1^{<r>}
\ar[r]^{d_1^{<r>}}& L_{0}^{<r>} \ar[r]& 0},$$ where
$$L_i^{<r>} =R(-i-r)^{\beta_{i, i+r}(I)},$$
and the differential $d_i^{<r>}: L_i^{<r>}\longrightarrow L_{i-1}^{<r>}$
is the corresponding component of $d_i$.

\vskip 2mm

\begin{remark}\label{L^r and L^{r-1}}
Sometimes we will also consider the minimal $\bZ$-graded free
resolution $L_{\bullet}(R/I)$ of the quotient ring $R/I$:
\begin{equation}\label{resolution of R/I}
L_{\bullet}(R/I): \hskip 3mm \xymatrix{ 0 \ar[r]& L_{n}
\ar[r]^{d_{n}}& \cdots \ar[r]& L_1 \ar[r]^{d_1}& L_{0}=R \ar[r]& R/I
\ar[r]& 0},
\end{equation}
 Its
truncation at the first term $L_{\ge 1}(R/I)$  gives a minimal free
resolution $L_\bullet(I)$  of $I$. For $r \ge 2$,
$\LL^{<r>}_\bullet(I)$ is isomorphic to  the $(r-1)$-linear strand
$\LL^{<r-1>}_\bullet(R/I)$   up to translation. However, this is not
true for $r=1$, since  $\LL^{<0>}_\bullet(R/I)$ starts from the
$0$-th term $R$, which is irrelevant to $\LL_\bullet^{<1>}(I)$.
\end{remark}

\vskip 2mm

To the minimal $\bZ$-graded free resolution of $I$ we may associate a set of invariants that
measure the acyclicity of the linear strands as follows:
Let $\KK$ denote the field of fractions $Q(R)$ of $R$, and set
$$\nu_{i,j}(I):= \dim_{\KK} [H_i(\mathbb{L}_{\bullet}^{<j-i>}(I) \otimes_R \KK)].$$
Since the complex $\mathbb{L}_{\bullet}^{<r>}(I) \otimes_R \KK$ is of the form
$$\xymatrix{ 0 \ar[r]&
\KK^{\beta_{n,n+r}(I)} \ar[r]^{\qquad  \partial_{n}^{<r>}}& \cdots \ar[r]& \KK^{\beta_{1,1+r}(I)}
\ar[r]^{\quad \partial_1^{<r>}}&  \KK^{\beta_{0,r}(I)} \ar[r]& 0},$$
we have $\nu_{i,j}(I) \le \beta_{i,j}(I)$ for all $i,j$ (if $i > j$ then $\nu_{i,j}(I) = \beta_{i,j}(I) =0$),
and
$$\sum_{i=0}^n (-1)^i \nu_{i, i+r}(I)=\sum_{i=0}^n (-1)^i \beta_{i, i+r}(I)$$
for each $r$. If we mimic the construction of the {\it Betti table}, we may also consider
the {\it $\nu$-table} of $I$

\begin{center}
  \begin{tabular}{c|rrrrr}
    $ \nu_{i, i+r}(I) $ & $0$ & $1$ & $2$ & $\cdots$ \\ \hline
    $0$ &  $\nu_{0,0}(I)$ & $\nu_{1,1}(I)$ & $\nu_{2,2}(I)$ & $\cdots$  \\
    $1$ & $\nu_{0,1}(I)$ & $\nu_{1,2}(I)$ & $\nu_{2,3}(I)$ & $\cdots$  \\
   $\vdots$ & $\vdots$ & $\vdots$ & $\vdots$ &
  \end{tabular}
  \end{center}

  \vskip 2mm

Next we consider some basic properties of $\nu$-numbers. It turns out that they satisfy
analogous properties to those of Lyubeznik numbers. For instance, these invariants satisfy the
following {\it Euler characteristic} formula.

\begin{lemma}\label{Eul char gamma}
For a $\bZ$-graded ideal $I$, we have
$$\sum_{i,j \in \bN} (-1)^i\nu_{i,j}(I)=1.$$
\end{lemma}

\begin{proof}
The assertion follows from the computation below.
\begin{eqnarray*}
&&\sum_{i,j \in \bN} (-1)^i\nu_{i,j}(I)\\
&=& \sum_{i, r \in \bN}  (-1)^i\nu_{i,i+r}(I)\\
&=& \sum_{r \in \bN} \sum_{0 \le i \le n}  (-1)^i\nu_{i,i+r}(I)\\
&=& \sum_{r \in \bN} \sum_{0 \le i \le n}  (-1)^i\beta_{i,i+r}(I) \\ 
&=& \sum_{0 \le i \le n} \sum_{r \in \bN}  (-1)^i\beta_{i,i+r}(I)\\
&=& \sum_{0 \le i \le n}  (-1)^i\beta_i(I)\\
&=&1
\end{eqnarray*}
\end{proof}

We can also single out a particular non-vanishing $\nu$-number.
For each $i \in \bN$, let $I_{<i>}$ denote the ideal generated by the homogeneous component $I_i= \{ f \in I \mid \deg(f)=i \} \cup \{ 0 \}$.
Then we have:

\begin{lemma}\label{gamma_{o,d}}
If $I$ is a $\bZ$-graded ideal with $l:= \min\{ i \mid I_i \ne 0 \}$, then we have $\nu_{0,l}(I) \ne 0$.
\end{lemma}

\begin{proof}
It is easy to see that there is a surjection $H_0(\mathbb{L}_{\bullet}^{<l>}(I)) \twoheadrightarrow I_{<l>}$.
Since $\dim_R I_{<l>} =n$, we have $H_0(\mathbb{L}_{\bullet}^{<l>}(I)\otimes_R \KK)  \cong
H_0(\mathbb{L}_{\bullet}^{<l>}(I)) \otimes_R \KK \ne 0$.
\end{proof}

\vskip 2mm

This fact allows us to consider the following notion:

\begin{definition}
Let $I$ be a $\bZ$-graded ideal and set $l:= \min\{ i \mid I_i \ne 0 \}$.
We say that $I$ has {\it trivial $\nu$-table}, if  $\nu_{0,l}(I) =1$ and
the rest of these invariants vanish.
\end{definition}

\vskip 2mm

\subsection{Componentwise linear ideals}

It might be an interesting problem to find necessary and/or sufficient conditions for a $\bZ$-graded ideal to have a
trivial $\nu$-table. In this direction we have the following relation to the notion of {\it componentwise linear}
ideals.

\vskip 2mm

\begin{definition}[{Herzog and Hibi, \cite{HH}}]
We say a $\bZ$-graded ideal $I$ is {\it componentwise linear} if $I_{<r>}$ has a linear resolution for all $r \in \bN$,
i.e., $\beta_{i,j}(I_{<r>})=0$ unless $j=i+r$.
\end{definition}

R\" omer (\cite{Rom01}) and the second author (\cite[Theorem~4.1]{Yan00}) independently showed that $I$ is  componentwise linear
if and only if $H_i({\mathbb L}_\bullet^{<r>}(I))=0$ for all $r$ and all $i \ge 1$.

\begin{proposition}\label{comp linear => trivial gamma}
A componentwise linear ideal $I$ has  a trivial $\nu$-table.
\end{proposition}

\begin{proof}
Since $I$ is componentwise linear, we have $H_i(\mathbb{L}_{\bullet}^{<r>}(I)) =0$ for all $r$ and all $i \ge 1$, and hence
$\nu_{i,j}(I)=0$  for all $j$ and all $i \ge 1$. Now the assertion follows from Lemmas~\ref{Eul char gamma} and \ref{gamma_{o,d}}.
\end{proof}

The converse of the above proposition is not true.
For example, in Corollary~\ref{I_1 ne 0} below, we will show that if $I_1 \ne 0$  then it has trivial $\nu$-table.
However, there is no relation between being componentwise linear and $I_1 \ne 0$.

\subsection{Consecutiveness of nontrivial columns of the $\nu$-tables}
For a $\bZ$-graded ideal $I \subseteq R$ and $i \in \bN$, set
$$\nu_i(I)=\sum_{j \in \bN} \nu_{i,j}(I).$$
If we denote $\LL_\bullet(I):= \bigoplus_{r \in \bN} \mathbb{L}_\bullet^{<r>}(I)$, then
$$\nu_i(I) = \dim_\KK H_i(\LL_\bullet(I) \otimes_R \KK).$$
By Lemma~\ref{Eul char gamma}, we have $\sum_{i=0}^n(-1)^i\nu_i(I)=1$.
We say $\nu_i(I)$ is {\it non-trivial}, if
$$
\nu_i(I) \ge
\begin{cases}
2 & \text{if $i=0$,} \\
1 & \text{if $i \ge 1$.}
\end{cases}
$$
Clearly, the $\nu$-table of $I$ is non-trivial if and only if $\nu_i(I)$ is non-trivial for some $i$.
If $n \ge 1$,  we have ${\rm proj.dim}_R I \le n-1$, and hence $\nu_n(I)=0$. In particular, $\nu_n(I)$ is always trivial.

\vskip 2mm

The main result of this subsection is the following:

\begin{theorem}\label{consecutive gamma}
Let $I$ be a $\bZ$-graded ideal of $R$. Then;
\begin{itemize}
\item If $\nu_j(I)$ is non-trivial for $1 \le  j \le n-1$, then either $\nu_{j-1}(I)$ or  $\nu_{j+1}(I)$ is non-trivial.
\item  If $\nu_0(I)$ is non-trivial, then so is $\nu_1(I)$.
\end{itemize}
\end{theorem}

In order to prove the theorem, we will reconstruct  $\LL_\bullet(I)$ using a spectral sequence.
Let $L_\bullet(I)$ be the minimal free  resolution of  $I$ as before.
Consider the $\fM$-adic filtration
$L_\bullet(I) = F_0 L_\bullet \supset F_1 L_\bullet \supset \cdots$ of
$L_\bullet(I)$, where $F_iL_\bullet$ is a subcomplex   whose component of homological degree $j$ is $\fM^i L_j$.
For any given $R$-module $M$,  we regard  $\gr (M) := \bigoplus_{i \in \bN} \fM^i M/\fM^{i+1} M$ as an $R$-module
via the isomorphism $\gr R = \bigoplus_{i \in \bN} \fM^i/\fM^{i+1} \cong R =\kk[x_1, \ldots, x_n]$.
Since each $L_j$ is a free $R$-module, $$\bigoplus_{p+q=-j}E_0^{p,q} =
\left(\bigoplus_{p \geq 0} \fM^p L_j / \fM^{p+1} L_j \right)=  \gr L_j$$
is isomorphic to $L_j$ (if we identify $\gr R$ with $R$), while we have to forget the original $\bZ$-grading of $L_j$.
Since $L_\bullet(I)$ is a minimal free resolution,
$d_0^{p,q} : E_0^{p,q} \to E_0^{p,q+1}$ is the  zero map for all $p,q$,
and hence $E_0^{p,q} = E_1^{p,q}$. It follows that
$$\EE^{(1)}_j := \bigoplus_{p+q=-j}E_1^{p,q} = \bigoplus_{p+q=-j}E_0^{p,q}$$
is isomorphic to $L_j$ under the identification $R \cong \gr R$.
Collecting the maps $$d_1^{p,q}: E_1^{p,q} (= \fM^p L_j/\fM^{p+1} L_j)
\longrightarrow E_1^{p+1,q} (= \fM^{p+1} L_{j-1}/\fM^{p+2} L_{j-1})$$ for $p,q$ with $p+q=-j$, we have the $R$-morphism
$d^{(1)}_j:\EE^{(1)}_j  \to \EE^{(1)}_{j-1}$,  and these morphisms  make
$\EE^{(1)}_\bullet$ a chain complex of $R$-modules. 
Under the isomorphism $\EE^{(1)}_j \cong L_j$,
$\EE^{(1)}_\bullet$ is isomorphic to $\LL_\bullet(I) = \bigoplus_{r \in \bN} \mathbb{L}_\bullet^{<r>}(I)$.
Hence we have
$$\EE^{(2)}_j:=\bigoplus_{p+q=-j} E_2^{p,q} \cong H_j(\LL_\bullet (I))$$
and $\nu_j (I) =\dim_\KK (\EE^{(2)}_j \otimes_R \KK)$.
Collecting the maps  $d_2^{p,q}: E_2^{p,q}  \to E_2^{p+2,q-1}$, we have the $R$-morphism
$$d^{(2)}_j:\EE^{(2)}_j \,  (\cong H_j(\LL_\bullet(I)))\longrightarrow \EE^{(2)}_{j-1} \, (\cong H_{j-1}(\LL_\bullet(I)).$$
 Moreover, we have  the chain complex
$$\cdots \longrightarrow \EE^{(2)}_{j+1} \stackrel{d^{(2)}_{j+1}}{\longrightarrow}  \EE^{(2)}_j  \stackrel{d^{(2)}_j}{\longrightarrow} \EE^{(2)}_{j-1} \longrightarrow \cdots,$$
of $R$-modules whose $j$th homology is isomorphic to $\EE_j^{(3)}:=\bigoplus_{p+q=-j} E_3^{p,q}.$
For all $r \ge 4$, $\EE^{(r)}_j:=\bigoplus_{p+q=-j} E_r^{p,q}$ satisfies the same property.

\vskip 2mm

By the construction of spectral sequences, if $$r > \max \{ k \mid \beta_{j,k}(I) \ne 0 \} - \min \{ k \mid \beta_{j-1,k}(I) \ne 0 \},$$
then the map $d_r^{p,q} : E_r^{p,q} \to E_r^{p+r, q-r+1}$ is zero for all $p,q$ with $p+q=-j$, and hence $d_j^{(r)} : \EE^{(r)}_j \to  \EE^{(r)}_{j-1}$ is zero.
It implies that $\EE_j^{(r)}$ is isomorphic to
$$\EE_j^{(\infty)}:= \bigoplus_{p+q=-j} E^{p,q}_\infty$$
for $r \gg 0$.

\medskip

\noindent{\it Proof of Theorem~\ref{consecutive gamma}.} We will
prove the assertion by contradiction using the spectral sequence
introduced above. First, assume that $\nu_j(I)$ is non-trivial
for some $2 \le  j \le n-1$, but both $\nu_{j-1}(I)$ and
$\nu_{j+1}(I)$ are trivial (the cases $j=0,1$ can be proved using
similar arguments, and we will give a few remarks later). Then we
have $\EE^{(2)}_j \otimes_R \KK \ne 0 $ and $\EE^{(2)}_{j+1}
\otimes_R \KK = \EE^{(2)}_{j-1} \otimes_R \KK =0.$ It follows that
$\EE^{(3)}_j \otimes_R \KK \ne 0,$ since it is the homology of
$$\EE^{(2)}_{j+1} \otimes_R \KK \longrightarrow \EE^{(2)}_j  \otimes_R \KK
 \longrightarrow \EE^{(2)}_{j-1} \otimes_R \KK.$$
Similarly, we have
$\EE^{(3)}_{j-1} \otimes_R \KK  = \EE^{(3)}_{j+1} \otimes_R \KK =0.$
Repeating this argument, we have $\EE^{(r)}_j \otimes_R \KK \ne 0$
for all $r \ge 4$. Hence $E_\infty^{p,q} \ne 0$ for some $p,q$ with $p+q=-j$.
However it contradicts the facts that
$$E_r^{p,q} \Longrightarrow H_{-p-q}(L_\bullet(I))$$
and $H_j(L_\bullet(I))=0$ (recall that $j >0$ now).

Next we assume that $\nu_1(I)$ is non-trivial, but $\nu_0(I)$ and $\nu_2(I)$ are trivial, that is,
$$\EE_1^{(2)} \otimes_R \KK \ne 0, \quad \EE_0^{(2)} \otimes_R \KK \cong \KK \quad \text{and} \quad \EE_2^{(2)} \otimes_R \KK = 0.$$
As we have seen above,  we must have $\EE_1^{(r)} = 0$ for  $r \gg 0$. 
Since $\EE_2^{(r)} \otimes_R \KK = 0$ for all $r$ now, if $d^{(r)}_1 \otimes_R \KK : \EE_1^{(r)}  \otimes_R \KK \longrightarrow  \EE_0^{(r)} \otimes_R \KK$ are the zero maps for all $r$, 
then  $\EE_1^{(r)} \otimes_R \KK \cong \EE_1^{(2)} \otimes_R \KK \ne 0$ for all  $r$, and this is a contradiction. 
So  there is some $r \ge 2$ such that $d^{(r)}_1 \otimes_R \KK$ is not zero. If $s$ is the minimum among these $r$, 
$d^{(s)}_1 \otimes_R \KK : \EE_1^{(s)}  \otimes_R \KK \longrightarrow  (\EE_0^{(s)} \otimes_R \KK) \cong \KK$ is surjective.
Hence $\EE_0^{(r)} \otimes_R \KK = 0$ for all $r >s$, and   $\EE_0^{(\infty)} \otimes_R \KK = 0$.
However,  since $\EE^{(\infty)}_0 \cong \gr (H_0(L_\bullet(I)) \cong \gr(I)$  and  $\dim_R I=n$,
we have $\dim_R  (\gr(I)) =n$ and hence  $\EE_0^{(\infty)} \otimes_R \KK \ne 0$. This is a contradiction.
The case when $\nu_0(I)$ is non-trivial can be proved in a similar way.
\qed

\subsection{Thom-Sebastiani type formulae}

Let $I,J$ be $\bZ$-graded ideals in
two disjoint sets of variables, say
$I \subseteq R=\kk[x_1, \ldots, x_m]$ and  $J \subseteq S=\kk[y_1, \ldots, y_n]$.
The aim of this subsection is to describe the $\nu$-numbers of $IT+JT$, where
$T=R\otimes_{\kk} S=\kk[x_1, \ldots, x_m,y_1, \ldots, y_n]$, in terms of those of $I$
and $J$ respectively. When we just consider Betti numbers we have the following results
due to Jacques-Katzman \cite{JK05}.

\begin{proposition}[{c.f. \cite[Lemma~2.1]{JK05}}]\label{JacKat}
Let $L_{\bullet}(R/I)$ and  $L_{\bullet}(S/J)$ be
minimal graded free resolutions of $R/I$ and $S/J$ respectively.
Then, $$(L_{\bullet}(R/I) \otimes_{R} T)\otimes_T (L_{\bullet}(S/J)\otimes_{S} T)$$ is a minimal graded free
resolution of $T/IT+JT$.
\end{proposition}

Hence, Betti numbers satisfy the following relation:

\begin{corollary}[{c.f. \cite[Corollary~2.2]{JK05}}]
The Betti numbers of $T/IT+JT$  have the following form:
$$\beta_{i,j}(T/IT+JT)=
 \sum_{\substack{k+k'=i\\l + l'= j}}
\beta_{k,l}(T/IT) \beta_{k',l'}(T/JT).$$
Hence we have 
$$\beta_{i,j}(IT+JT)=\beta_{i,j}(IT)+ \beta_{i,j}(JT)
+  \sum_{\substack{k+k'=i-1\\l + l'= j}}
\beta_{k,l}(IT) \beta_{k',l'}(JT).$$

\end{corollary}

Our aim is to extend the result in \cite{JK05} to the case of $\nu$-numbers. To such purpose it will be more convenient to
consider separately the case of ideals with degree one elements. Thus, let $I \subseteq R$ be
any $\bZ$-graded ideal and assume for simplicity
that $J$ is principally generated by an element of degree one, e.g. $J=(y) \subseteq S$.

\begin{lemma}\label{mapping cone}
Let $I \subseteq R=\kk[x_1, \ldots, x_m]$ and  $J=(y)\subseteq S=\kk[y]$ be $\bZ$-graded
ideals and set $T=R\otimes_{\kk} S=\kk[x_1, \ldots, x_m,y]$. For $r \ge 2$, the $r$-linear strand
$\LL_\bullet^{<r>}(IT+JT)$ is the mapping cone of  the chain map
$$\times y: (\LL_\bullet^{<r>}(IT))(-1) \to
\LL_\bullet^{<r>}(IT).$$
\end{lemma}

\begin{proof}
It is easy to see that a minimal $T$-free resolution $L_\bullet(T/IT+JT)$ of $T/IT+JT$ is given by the
mapping cone of the chain map $\times y: L_\bullet(T/IT)(-1) \to L_\bullet(T/IT)$, where
$L_\bullet(T/IT)$ is a minimal $T$-free resolution of $T/IT$.
Since the operation of taking $r$-linear strand commutes with the operation of taking the mapping cone, we are done.
\end{proof}

The general case is more involved. Assume now that $I \subseteq R=\kk[x_1, \ldots, x_m]$ and  
$J \subseteq S=\kk[y_1, \ldots, y_n]$ are $\bZ$-graded ideals 
such that $I_1=0$ and $J_1=0$. Let  $L_{\bullet}(I)$ be a minimal graded  $R$-free
resolution of $I$ and $L_{\bullet}(J)$ a minimal graded  $S$-free
resolution of $J$ and consider their extensions $L_{\bullet}(IT)$ and $L_{\bullet}(JT)$ to 
$T=R\otimes_{\kk} S=\kk[x_1, \ldots, x_m,y_1, \ldots, y_n]$.

\begin{lemma}\label{tensor}
Under the previous assumptions,  the $r$-linear strand
$\LL_\bullet^{<r>}(IT+JT)$ is
$$\LL_{\bullet}^{<r>}(IT+JT)= \>\LL_{\bullet}^{<r>}(IT)\oplus \LL_{\bullet}^{<r>}(JT)
\oplus \left( \bigoplus_{a+b=r+1} (\LL_{\bullet}^{<a>}(IT)
\otimes_{T}\LL_{\bullet}^{<b>}(JT))[-1] \right).$$ 
Here, for a chain complex $C_\bullet$, $C_\bullet[-1]$ denotes 
the translated complex whose component of homological degree $j$  is $C_{j-1}$.
\end{lemma}

\begin{proof}
Consider the minimal $\bZ$-graded free resolutions of $R/I$ and
$S/J$ respectively

$L_{\bullet}(R/I): \hskip 3mm \xymatrix{ 0 \ar[r]& L_{m}
\ar[r]^{d_{m}}& \cdots \ar[r]& L_1 \ar[r]^{d_1}& L_{0} \ar[r]& R/I
\ar[r]& 0},$

$L_{\bullet}'(S/J): \hskip 3mm \xymatrix{ 0 \ar[r]&
{L'}_{n} \ar[r]^{d'_n}& \cdots \ar[r]& {L'}_1 \ar[r]^{d'_1}&
{L'}_{0} \ar[r]& S/J \ar[r]& 0}, $

\vskip 2mm

\noindent where ${L}_{0}=R$ and ${L'}_{0}=S$. According to
Proposition \ref{JacKat}, the minimal $\bZ$-graded free resolution
$L_{\bullet}(T/IT+JT)$ has the form\footnote{By an abuse of
notation we denote $(L_{i}\otimes_{R} T)\otimes_T
(L'_{j}\otimes_{S} T)$ simply as $L_{i}\otimes L'_{j}$}:
$$
  \xymatrix{\cdots \ar[r]& {\begin{array}{c}
  L_2 \otimes L'_0 \\
  \op \\
  L_1\otimes L'_1 \\
  \op \\
   L_0\otimes L'_2
\end{array}} \ar[r]^{\partial_2}& {\begin{array}{c}
  L_1\otimes L'_0 \\
  \op \\
  L_0\otimes L'_1
\end{array}} \ar[r]^{\partial_1}&
L_0\otimes L'_0\ar[r] & T/IT+JT\ar[r] &0 },$$ where, for any given
$x_i \otimes y_{p-i}\in L_{i}\otimes L'_{p-i}$, we have
$$\partial_p(x_i \otimes y_{p-i})=d_i(x_i)\otimes y_{p-i} + (-1)^i
x_i \otimes d'_{p-i}(y_{p-i}) \in (L_{i-1}\otimes L'_{p-i})\oplus
(L_{i}\otimes L'_{p-i-1}).$$

\vskip 2mm

To describe the $r$-linear strand
$\LL^{<r>}_{\bullet}(IT+JT)$ of the ideal $IT+JT$ we must
consider the truncation at the first term of the above resolution
and take a close look at the free modules and the components of the
corresponding differentials.
Recall that $L_\bullet^{<r-1>}(R/I)$ corresponds to $L_\bullet^{<r>}(I)$ for all $r \ge 2$.
It is easy to see that  both
$$\LL^{<r>}_{\bullet}(IT):
 0 \longrightarrow L^{<r-1>}_m \otimes L'_{0} \longrightarrow \cdots  \longrightarrow
L^{<r-1>}_2 \otimes L'_{0}  \longrightarrow L^{<r-1>}_1\otimes L'_{0}  \longrightarrow  0$$
and
$$\mathbb{L}^{<r>}_{\bullet}(JT):  0 \longrightarrow
L_0 \otimes {L'}^{<r-1>}_n \longrightarrow \cdots \longrightarrow
L_0\otimes {L'}^{<r-1>}_{2} \longrightarrow L_0\otimes {L'}^{<r-1>}_{1} \longrightarrow  0$$
are subcomplexes of  $\LL^{<r>}_{\bullet}(IT+JT)$.
Moreover, $\LL^{<r>}_{\bullet}(IT)$ and  $\LL^{<r>}_{\bullet}(JT)$ are direct summands of  $\LL^{<r>}_{\bullet}(IT+JT)$.
In fact, since $I_1=J_1=0$, the linear parts of the maps $L_i \otimes L'_1 \to L_i \otimes L'_0$
and $L_1 \otimes L'_j \to L_0 \otimes L'_j$ vanish.

In order to obtain the remaining components of $\LL^{<r>}_{\bullet}(IT+JT)$ we must consider the
$r$-linear strand of
$$
  \xymatrix{\cdots \ar[r]& {\begin{array}{c}
  L_3 \otimes L'_1 \\
   \op \\
   L_2 \otimes L'_2 \\
   \op \\
   L_1\otimes L'_3
\end{array}} \ar[r]^{\partial_4}   & {\begin{array}{c}
  L_2 \otimes L'_1 \\
   \op \\
   L_1\otimes L'_2
\end{array}} \ar[r]^{\partial_3}&
  L_1\otimes L'_1  \ar[r]^(.6){\partial_2}&
0\ar[r]^{\partial_1} & 0 }.$$
This complex starts at the second term (i.e.,  the term of homological degree 1),  and the first term of the $r$-linear strand is
$\bigoplus_{a+b=r+1} \LL_0^{<a>}(IT)\otimes_{T}\LL_0^{<b>}(JT).$
If we take a close look at the free summands of  these components
and its differentials we obtain the following description
$$\bigoplus_{a+b=r+1} (\LL_{\bullet}^{<a>}(IT)
\otimes_{T}\LL_{\bullet}^{<b>}(JT))[-1].$$
So we are done.
\end{proof}

The main result of this subsection is the following:

\begin{proposition}\label{join}
The $\nu$-numbers of $IT + JT$  have the following form:

\begin{itemize}
\item[i)]   If $I_1\neq 0$ or $J_1 \neq 0$  then $IT+JT$ has trivial $\nu$-table.

\item[ii)]   If $I_1 = 0$ and $J_1 = 0$ then we have:
$$\nu_{i,j}(IT + JT)= \nu_{i,j}(IT) + \nu_{i,j}(JT) +
\sum_{\substack{k+k'=i-1\\l + l'= j}} \nu_{k,l}(IT)
\nu_{k', l'}(JT).$$
\end{itemize}
\end{proposition}

\begin{proof}

i) If $J_1 \neq 0$, we may assume that $y_n \in J$ without loss of generality. 
Now we have $J=(f_1, \ldots, f_r, y_n)$, where $f_1 \ldots, f_r$ are homogeneous polynomials in $\kk[y_1, \ldots. y_{n-1}]$. 
Set $R' :=  \kk[x_1, \ldots, x_m, y_1, \ldots, y_{n-1}]$, $S'=\kk[y_n]$, and let $I'= IR' + (f_1, \ldots, f_r)$ be an ideal in $R'$ (note that 
$f_1 \ldots, f_r$ are elements in $R'$), and $J'=(y_n)$ an ideal in $S'$. Then we have $T= R \otimes_\kk S = R' \otimes_\kk S'$, and 
$IT+JT=I'T+J'T$.  This means that we may assume that  $J=(y)\subseteq S=\kk[y]$ from the beginning. 
For $r \ge 2$, the $r$-linear
strand $\LL^{<r>}_\bullet(IT+JT)$ is given by the mapping cone of
the chain map $\times y : (\LL_\bullet^{<r>}(IT))(-1) \to
\LL_\bullet^{<r>}(IT)$ by Lemma~\ref{mapping cone}.  Hence
$\LL^{<r>}_\bullet(IT+JT) \otimes_T \KK$ is given by the mapping
cone of the chain map
$$\times y : \LL_\bullet^{<r>}(IT)  \otimes_T \KK \longrightarrow \LL_\bullet^{<r>}(IT) \otimes_T \KK,$$
where $\KK$ is the field of fractions of $T$. 
Clearly, this is the identity map, and  its mapping cone is exact. It means that
$H_i(\LL^{<r>}_\bullet(IT+JT) \otimes_T \KK)=0$ for all $r \ge 2$ and all $i$.  

On the other hand, $(IT+JT)_{<1>}$ is a complete intersection ideal generated by degree 1 elements, and hence 
 we have  $\dim_\KK  H_i(\LL_\bullet^{<1>}(I) \otimes_T \KK) = \delta_{0,i}$.  Summing up, we see that $IT+JT$ has trivial $\nu$-table.

 \vskip 2mm

 ii) Follows immediately from Lemma~\ref{tensor}.
\end{proof}

The following is just a rephrasing of part i) of the previous result.

\begin{corollary}\label{I_1 ne 0}
Let  $I \subseteq R$ be  a $\bZ$-graded ideal with $I_1 \ne 0$, then $I$ has trivial $\nu$-table.
\end{corollary}

The following is another corollary of Proposition~\ref{join}.

\begin{corollary}\label{sum ==> nontirivial}
With the same notation as in Proposition~\ref{join}, if $I_1=J_1=0$, then 
$IT+JT$ always has non-trivial $\nu$-table. 
\end{corollary}

\begin{proof}
Set $l:= \min\{ i \mid I_i  \ne 0 \}$ and  $l':= \min\{ i \mid J_i  \ne 0 \}$. 
Then we have $\nu_{1, l+l'}(IT+JT) \ge \nu_{0,l}(IT)\nu_{0,l'}(JT) >0$ by Proposition~\ref{join} ii).  
\end{proof}

\section{Lyubeznik numbers vs $\nu$-numbers for monomial ideals}

In  \cite{Yan01}, the second author showed that, via Alexander
duality, the study of local cohomology modules with supports in
monomial ideals can be ``translated'' into the study of the minimal
free resolutions of squarefree monomial ideals. This fact was later
refined by A.~Vahidi and the first author in \cite{AV11} in order to
study Lyubeznik numbers of squarefree monomial ideals in terms of
the linear strands of its Alexander dual ideals. The aim of this
section is to go further in this direction.

\vskip 2mm

In the sequel we will only consider monomial ideals in the
polynomial ring $R=\kk[x_1, \ldots, x_n]$ and $\fM=(x_1, \ldots, x_n)$ will
denote the graded maximal ideal. Recall that Lyubeznik numbers are 
well define in this non-local setting since they are invariant with respect to 
completion so we consider $\lambda_{p,i} (R/I)=\lambda_{p,i} (\widehat{R}/I\widehat{R})$
where $\widehat{R}=\kk[\![x_1, \ldots, x_n]\!]$. For a vector $\ba=(a_1,
\ldots, a_n) \in \bN^n$, set $\supp(\ba) :=\{ i \mid a_i \ne 0\}
\subseteq \{1, \ldots, n \}$. For each $1 \le i \le n$,  let  $\be_i
\in \bZ^n$ be the $i$th standard vector. The following  notion was
introduced  by the second author, and serves a powerful tool for
combinatorial commutative algebra.

\begin{definition}\label{sqf module for poly ring}
We say a finitely generated $\bN^n$-graded $R$-module $M= \bigoplus_{\ba \in \bN^n} M_{\ba}$ is {\it squarefree}, if
the multiplication maps $M_{\ba} \ni y \longmapsto x_iy \in M_{\ba +\be_i}$ is bijective for all
$\ba \in \bN^n$ and all $i \in \supp (\ba)$.
\end{definition}

The theory of squarefree modules is found in \cite{Yan00, Yan01, Yan03, Yan04}.
Here we list some basic properties.

\begin{itemize}
\item For a monomial ideal $I$, it is a squarefree $R$-module  if and only if  $I =\sqrt{I}$ (equivalently,
the Stanley-Reisner ideal $I_\Delta$ for some $\Delta$).
The free modules $R$ itself and the $\bZ^n$-graded canonical module $\omega_R=R(-{\mathbf 1})$ are squarefree.
Here ${\mathbf 1} = (1,1, \ldots, 1) \in \bN^n$.
The Stanley-Reisner ring $R/I_\Delta$ is also squarefree.

\medskip

\item Let $M$ be a squarefree $R$-module, and $L_\bullet$ its $\bZ^n$-graded minimal free resolution.
Then the  free module $L_i$ and the syzygy module $\operatorname{Syz}_i(M)$ are squarefree for each $i$.
Moreover, $\Ext_R^i(M,\omega_R)$ is squarefree for all $i$.

\medskip

\item Let $\gmod R$ be the category of $\bZ^n$-graded finitely generated $R$-modules, and
$\Sq R$ its full subcategory consisting of squarefree modules. Then $\Sq R$ is an abelian subcategory of $\gmod R$.
We have an exact contravariant functor $\bA$ from $\Sq R$ to itself.
The construction of $\bA$ is found in (for example) \cite{Yan04}. Here we just remark that $\bA(R/I_\Delta) \cong I_{\Delta^\vee}$,
where $\Delta^\vee := \{ F \subseteq \{1, \ldots,n \} \mid (\{1, \ldots,n \} \setminus F) \not \in \Delta \}$ is the
Alexander dual simplicial complex of $\Delta$.

\end{itemize}

In this framework we have the following description of Lyubeznik
numbers.

\begin{theorem}[{\cite[Corollary~3.10]{Yan01}}]\label{Yan01, 3.10}
Let $R=\kk[x_1, \ldots, x_n]$ be a polynomial ring, and $I_\Delta$ a
squarefree monomial ideal.  Then we have
$$\lambda_{p,i}(R/I_\Delta) = \dim_\kk [\Ext_R^{n-p}(\Ext_R^{n-i}(R/I_\Delta, \omega_R), \omega_R)]_0< \infty.$$
\end{theorem}

For a squarefree $R$-module $M$, the second author  defined the cochain complex $\bD(M)$ of squarefree $R$-modules satisfying
$H^i(\bD(M)) \cong \Ext^{n+i}_R(M, \omega_R)$ for all $i$ (see \cite[\S 3]{Yan04}).
By \cite[Theorem~4.1]{Yan00} or \cite[Theorem~3.8]{Yan03}, we have
the isomorphism
\begin{equation}\label{Yan01, Thm 4.1}
\bA \circ \bD(\Ext_R^{n-i}(R/I_\Delta, \omega_R)) \cong (\mathbb{L}_\bullet^{<n-i>}(I_{\Delta^\vee}))[-i]
\end{equation}
of cochain complexes of $\bZ^n$-graded $R$-modules\footnote{Our situation is closer to that of \cite[Theorem~4.1]{Yan00}
(\cite{Yan04} works in  a wider context).
However, \cite{Yan00} does not recognize $\bD$ and $\bA$  as individual operations, but treats the composition $\bA \circ \bD$.
In fact, $\bA \circ \bD$ corresponds to the operation ${\mathbb F}_\bullet(-)$ of \cite{Yan00} up to translation.}.
Here, for a cochain complex $C^\bullet$, $C^\bullet[-i]$ means the $-i$th translation of $C^\bullet$, more precisely,
it is the cochain complex whose component of cohomological degree $j$  is $C^{j-i}$, and 
we regard a chain complex $C_\bullet$ as the cochain complex whose component of cohomological degree $j$  is $C_{-j}$.

\vskip 2mm

The following is a variant of a result given by the first author and  A.~Vahidi.

\begin{theorem}[{c.f. \cite[Corollary~4.2]{AV11}}]\label{lambda vs gamma}
  Let $I_\Delta \subseteq R=\kk[x_1, \ldots, x_n]$ be a squarefree monomial ideal. Then we have
$$\lambda_{p,i}(R/I_\Delta) = \nu_{i-p, n-p}(I_{\Delta^\vee}).$$
\end{theorem}

\begin{proof}
By \eqref{Yan01, Thm 4.1} and the construction of $\bA$,
we have an isomorphism
$$([\bD(\Ext^{n-i}_R(R/I_\Delta, \omega_R)]_0)^* \cong (\mathbb{L}_\bullet^{<n-i>}(I_{\Delta^\vee}))_{\mathbf 1}[-i]$$
of cochain complexes of $\kk$-vector spaces. Here $(-)^*$ means the $\kk$-dual.
We also remark that, for a squarefree module $M$, we have $$\dim_\kk M_{\mathbf 1}=\operatorname{rank} _R M=\dim_\KK M \otimes_R \KK.$$
Thus we have the following computation.
\begin{eqnarray*}
\lambda_{p,i}(R/I_\Delta) &= & \dim_\kk [\Ext_R^{n-p}(\Ext_R^{n-i}(R/I, \omega_R), \omega_R))]_0\\
&=& \dim_\kk [H^{-p}(\bD(\Ext_R^{n-i}(R/I, \omega_R) )]_0\\
&=& \dim_\kk [H_{i-p}(\mathbb{L}_\bullet^{<n-i>}(I_{\Delta^\vee}))]_{\mathbf 1}\\
&=& \dim_\KK H_{i-p}(\mathbb{L}_\bullet^{<n-i>}(I_{\Delta^\vee})) \otimes_R \KK\\
&=& \nu_{i-p, n-p}(I_{\Delta^\vee}).
\end{eqnarray*}
\end{proof}

As mentioned in Introduction, for a local ring $A$ containing a field, we have
$$\sum_{0 \le p,i \le n} (-1)^{p-i} \lambda_{p,i}(A)=1.$$
In the monomial ideal case, this equation is an immediate consequence of Lemma~\ref{Eul char gamma} and
Theorem~\ref{lambda vs gamma}.

As a special case of Theorem~\ref{consecutiveness},  the Lyubeznik tables of monomial ideals in $R=\kk[x_1, \ldots, x_n]$ satisfy the consecutiveness property
of  nontrivial superdiagonals.  However, it also follows from
the consecutiveness property of nontrivial columns of the $\nu$-tables
(Theorem~\ref{consecutive gamma}) via Theorem~\ref{lambda vs gamma}.
In this sense, both ``consecutiveness theorems'' are related.

\subsection{Sequentially Cohen-Macaulay rings}
Let $M$ be a finitely generated graded module over the polynomial ring $R=\kk[x_1, \ldots, x_n]$.
We say $M$ is {\it sequentially Cohen-Macaulay} if $\Ext_R^{n-i}(M,R)$
is either a Cohen-Macaulay module of dimension $i$ or the 0 module for all $i$.
The original definition is given by the existence of a certain filtration
(see \cite[III, Definition~2.9]{Sta96}),
however it is equivalent to the above one by \cite[III, Theorem~2.11]{Sta96}.
The sequentially Cohen-Macaulay property of a finitely generated module over a regular local ring
is defined/characterized in the same way.

\vskip 2mm

In \cite{Alv13}, the first author showed that the sequentially Cohen-Macaulay property implies the triviality of Lyubeznik tables
in  positive characteristic as well as in the case of squarefree monomial ideals. 
Using Proposition \ref{comp linear => trivial gamma} we can give a
new proof/interpretation of this result for the case of monomial
ideals.

\begin{proposition}[{c.f. \cite[Theorem~3.2]{Alv13}}] \label{seq CM}
Let $I$ be a monomial ideal of the  polynomial ring $R=\kk[x_1, \ldots, x_n]$ such that $R/I$ is sequentially Cohen-Macaulay.
Then the Lyubeznik table of $R/I$ is trivial.
\end{proposition}

\begin{proof}
By  \cite[Theorem~2.6]{HTT}, $R/\sqrt{I}$ is  sequentially Cohen-Macaulay again.
Hence we may assume that $I$ is the Stanley-Reisner ideal $I_\Delta$ of a simplicial complex $\Delta$.
Herzog and Hibi \cite{HH} showed that $R/I_\Delta$ is sequentially Cohen-Macaulay if and only if $I_{\Delta^\vee}$ is componentwise linear.
Now the assertion immediately follows from Proposition~\ref{comp linear => trivial gamma} and Theorem~\ref{lambda vs gamma}.
\end{proof}

\vskip 2mm

The converse of Proposition~\ref{seq CM} is not true, that is, even if $R/I$ has trivial Lyubeznik table
it need not be sequentially Cohen-Macaulay.
For example, if $I$ is the monomial ideal
$$(x_1, x_2) \cap (x_3, x_4) \cap (x_1,x_5) \cap  (x_2,x_5) \cap (x_3,x_5) \cap  (x_4,x_5)$$
in $R=\kk[x_1, \ldots, x_5]$, then $R/I$ has  trivial Lyubeznik
table, but this ring is {\it not} sequentially Cohen-Macaulay. Since
all associated primes of $I$ have the same height, it is the same
thing to say $R/I$ is not Cohen-Macaulay. However, $R/I$ does not
even satisfy  Serre's condition $(S_2)$.

\vskip 2mm

In Proposition \ref{dual_join} below, we will see that if  a monomial
ideal $I$ has height one (i.e., admits a height one associated prime),  then the Lyubeznik table
of $R/I$ is trivial. Of course,  $R/I$ need not be sequentially
Cohen-Macaulay in this situation.

\subsection{Thom-Sebastiani type formulae}
Let $I \subseteq R=\kk[x_1, \ldots, x_m]$ and  $J \subseteq
S=\kk[y_1, \ldots, y_n]$ be squarefree monomial ideals in two
disjoint sets of variables. Let $\Delta_1$ and $\Delta_2$ be the
simplicial complexes associated to $I$ and $J$ by the
Stanley-Reisner correspondence, i.e. $I=I_{\Delta_1}$ and
$J=I_{\Delta_2}$. Then, the sum $IT+JT=I_{\Delta_1 \ast \Delta_2}$
corresponds to the simplicial {\it join} of both complexes.
Let $\Delta_1^\vee$ (resp. $\Delta_2^\vee$) be the Alexander dual of  $\Delta_1$ (resp. $\Delta_2$)
as a simplicial complex on $\{1,2, \ldots, m \}$ (resp.  $\{1,2, \ldots, n\}$).
Set $I^\vee := I_{\Delta_1^\vee} \subseteq R$ and $J^\vee := I_{\Delta_2^\vee} \subseteq S$.
Then it is easy to see that
$$\bA(T/IT) \cong I^\vee T,   \quad \bA(T/JT) \cong J^\vee T, \quad \text{and} \quad
\bA(T/IT \cap JT) \cong I^\vee T + J^\vee T,$$
where $\bA$ denotes the Alexander duality functor of $\Sq T$.

\vskip 2mm

\begin{proposition}\label{dual_join}
The Lyubeznik numbers of $T/IT\cap JT$  have the following form:

\begin{itemize}
\item[i)]   If either the height of $I$ or the height of $J$ is  $1$, then $T/IT \cap JT$ has trivial Lyubeznik table.

\item[ii)]  If both the height of $I$ and the height of $J$ are  $\geq 2$, then we have:
\begin{eqnarray*}
\lambda_{p,i}(T/IT \cap JT) &=& \lambda_{p,i}(T/IT) + \lambda_{p,i}(T/JT) +
\sum_{\substack{q+r=p+\dim T\\j+k=i+\dim T-1}} \lambda_{q,j}(T/IT) \lambda_{r,k}(T/JT)\\
&=& \lambda_{p-n,i-n}(R/I) + \lambda_{p-m,i-m}(S/J) +
\sum_{\substack{q+r=p\\j+k=i-1}} \lambda_{q,j}(R/I) \lambda_{r,k}(S/J).
\end{eqnarray*}
\end{itemize}
\end{proposition}


\begin{proof}
The assertion easily follows from Proposition \ref{join} and Theorem \ref{lambda vs gamma}, but for completeness, we will give a few remarks.

(i) Recall that, for a simplicial complex $\Delta$, the height of $I_\Delta$ is 1 if and only if $[I_{\Delta^\vee}]_1 \ne 0$.

(ii)
The last equality follows from the fact that $$\lambda_{p, i}(T/IT) = \lambda_{p-n, i-n}(R/I) \quad  \text{and} \quad
\lambda_{p, i}(T/JT) = \lambda_{p-m, i-m}(S/J),$$ which can be seen from Theorem \ref{lambda vs gamma}
and the construction of linear strands.
\end{proof}


\begin{example}
It is well-know that local cohomology modules as well as free
resolutions depend on the characteristic of the base field
so Lyubeznik numbers depend on the characteristic as well.
The most recurrent example is the Stanley-Reisner ideal associated to
a minimal triangulation of $\mathbb{P}_{\bR}^2$, i.e.  the ideal in $R=\kk[x_1,\dots, x_6]$:

\vskip 2mm

$I=(x_1x_2x_3,x_1x_2x_4,x_1x_3x_5,x_2x_4x_5,x_3x_4x_5,x_2x_3x_6,x_1x_4x_6,x_3x_4x_6,x_1x_5x_6,x_2x_5x_6).$

\vskip 2mm

\noindent Its Lyubeznik table has been computed in \cite[Ex. 4.8]{AV11}. Namely, in characteristic zero 
and two respectively, we have:

$$\Lambda_{\bQ}(R/I)  = \begin{pmatrix}
   0 & 0 & 0 & 0 \\
     & 0 & 0  & 0\\
     &   & 0 & 0 \\
&   &  & 1
\end{pmatrix} \hskip 5mm \Lambda_{\bZ/2\bZ}(R/I)  = \begin{pmatrix}
   0 & 0 & 1 & 0 \\
     & 0 & 0  & 0\\
     &   & 0 & 1 \\
&   &  & 1
\end{pmatrix}$$

One can slightly modify this example and use Proposition
\ref{dual_join} to obtain some interesting behavior of Lyubeznik
numbers:

\vskip 2mm

$\bullet$ The ideal $J= I \cap (x_7)$ in $R=\kk[x_1,\dots,
x_7]$  has trivial Lyubeznik table in any characteristic, so we
obtain an example where the local cohomology modules depend on the
characteristic but Lyubeznik numbers do not.

\vskip 2mm

$\bullet$ The ideal $J=I \cap (x_7,x_8)\cap (x_9,x_{10})$ in
$R=\kk[x_1,\dots, x_{10}]$ satisfies $$1=\lambda^{\bQ}_{6,7}(R/J)
\neq \lambda^{\bZ/2\bZ}_{6,7}(R/J)=2$$ and both Lyubeznik numbers are different from
zero.

\end{example}

\section{Lyubeznik table is a topological invariant}
While the other sections treat the case where $R$ is a regular local ring or a polynomial ring,
in this section we will work in a slightly different situation.
Here the ring $R$ means a normal semigroup ring.
When $R$ is simplicial and Gorenstein, the second author proved in \cite{Yan01-2} that the local cohomology modules
$H_I^r(R)$ have finite Bass numbers for radical monomial ideals $I \subset R$. In fact, without these conditions,
Bass numbers are out of control and can be infinite (see \cite{HM} for details).

\vskip 2mm

Before going to the main result of this section (Theorem \ref{topological}), we will introduce the setup on which we will
work with. For more details we refer to  \cite{Yan01-2}.

\vskip 2mm

Let $C \subset \bZ^n$ be an affine semigroup
(i.e., $C$ is a finitely generated additive submonoid of $\bZ^n$),
and $R := \kk[\bx^\bc \mid \bc \in C] \subset \kk[x_1^{\pm 1}, \ldots,
x_n^{\pm 1}]$ the semigroup ring of $C$ over $\kk$.
Here $\bx^\bc$ denotes the monomial $\prod_{i=1}^n x_i^{c_i}$
for $\bc = (c_1, \ldots, c_n) \in C$.
Regarding  $C$ as a subset of  $\bR^n = \bR \otimes_{\bZ} \bZ^n$,
let $P := \bR_{\geq 0} C \subset \bR^n$ be the polyhedral cone spanned by $C$.
We always assume that $\bZ C = \bZ^n$, $\bZ^n \cap P = C$ and $C \cap (-C) = \{  0 \}$.
Thus $R$ is a normal Cohen-Macaulay integral domain of dimension $n$
with the graded maximal ideal $\fM := (\bx^\bc \mid 0 \ne \bc \in C)$.
We say $R$ is {\it simplicial}, if the cone $P$ is spanned by $n$ vectors in $\bR^n$.
The polynomial ring $\kk[x_1, \ldots, x_n]$ is a typical example of a simplicial semigroup ring $\kk[C]$ for $C=\bN^n$.
Clearly, $R = \bigoplus_{\bc \in C} \kk \, \bx^\bc$ is a $\bZ^n$-graded ring.
We say that a $\bZ^n$-graded ideal of $R$ is a {\it monomial ideal} and we will denote $\gmod R$
 the category of finitely generated $\bZ^n$-graded $R$-modules and degree preserving $R$-homomorphisms.

\vskip 2mm

Let $L$ be the set of non-empty faces of the polyhedral cone $P$. Note that $\{ 0\}$ and $P$ itself belong to $L$.
Regarding $L$ as a partially ordered set by inclusion, $R$ is simplicial if and only if
$L$ is isomorphic to the power set $2^{\{1, \ldots, n\}}$.  For $F \in L$,
$\fp_F := ( \, \bx^\bc \mid \bc \in C \setminus  F \, )$ is a prime ideal
of $R$.  Conversely, any monomial prime ideal is of the
form $\fp_F$ for some $F \in L$. Note that 
$R/\fp_F \cong \kk[ \, \bx^\bc \mid \bc \in C \cap F]$ for $F \in L$.
For a point $\bc \in C$,  we always have a unique face
$F \in L$ whose relative interior contains $\bc$.
Here we denote $s(\bc) = F$.

\vskip 2mm

The following is a generalization of the notion of squarefree modules (see Definition~\ref{sqf module for poly ring})
to this setting.


\begin{definition}[\cite{Yan01-2}]\label{sq}
We say a module $M \in \gmod R$
is {\it squarefree}, if it is $C$-graded
(i.e., $M_\ba = 0$ for all $\ba \not \in C$),
and the multiplication map
$M_\ba \ni y \longmapsto \bx^\bb y
\in M_{\ba + \bb}$ is bijective for all
$\ba, \bb \in C$ with $s(\ba+\bb) = s(\ba)$.
\end{definition}

For a monomial ideal $I$, $R/I$ is a squarefree $R$-module
if and only if $I$ is a radical ideal (i.e., $\sqrt{I} = I$).
We say that $\Delta \subseteq L$ is an {\it order ideal} if
$\Delta \ni F \supset F' \in L$ implies $F' \in \Delta$.
If $\Delta$ is an order ideal, then
$I_\Delta := ( \, \bx^\bc \mid \bc \in C, \,
s(\bc) \not \in \Delta  \, ) \subseteq R$
is a radical monomial ideal.
Conversely, any radical monomial ideal is of the
form $I_\Delta$ for some $\Delta$. Clearly,
$$
[R/I_\Delta]_\bc \cong
\begin{cases}
\kk & \text{if $\bc \in C$ and $s(\bc) \in \Delta$,}\\
0 & \text{otherwise.}
\end{cases}
$$

If $R$ is simplicial, an order ideal $\Delta$ is essentially a simplicial complex on the vertices $1,2, \ldots, n$.
If  $R$ is  the polynomial ring $\kk[x_1, \ldots, x_n]$, then $R/I_\Delta$ is nothing but the Stanley-Reisner
ring of the simplicial complex $\Delta$.

\vskip 2mm

For each $F \in L$, take some $\bc(F) \in C \cap \relint(F)$
(i.e., $s(\bc(F)) =F$). For a squarefree $R$-module $M$ and $F, G \in L$ with $G \supset F$, \cite[Theorem~3.3]{Yan01-2} gives a
$\kk$-linear map $$\varphi^M_{G, F}: M_{\bc(F)} \to M_{\bc(G)}.$$
These maps satisfy $\varphi^M_{F,F} = \operatorname{Id}$ and $\varphi^M_{H, G} \circ \varphi^M_{G, F} = \varphi^M_{H,F}$ for all
$H \supset G \supset F$. We have $M_\bc \cong M_{\bc'}$ for
$\bc, \bc' \in C$ with $s(\bc) = s(\bc')$. Under these isomorphisms,
the maps  $\varphi^M_{G, F}$ do not depend on the particular choice of $\bc(F)$'s.

\vskip 2mm

Let $\Sq R$ be the full subcategory of $\gmod R$
consisting of squarefree modules. As shown in \cite{Yan01-2},
$\Sq R$ is an abelian category with enough injectives.
For an indecomposable squarefree module $M$,
it is injective in $\Sq R$ if and only if
$M \cong R/\fp_F$ for some $F \in L$.

\vskip 2mm

Let $\omega_R$ be the $\bZ^n$-graded canonical module of $R$.
It is well-known that $\omega_R$ is isomorphic to the radical
monomial ideal $(\, \bx^\bc \mid \bc \in C,  s(\bc)= P \, )$.
As shown in \cite[Proposition~3.7]{Yan01-2} we have $\Ext^i_R(M,\omega_R) \in \Sq R$ for $M \in \Sq R$.


\subsection{Lyubeznik numbers}
Let $R=\kk[C]$ be a normal simplicial semigroup ring which is Gorenstein, and $I$ a monomial ideal of $R$.
 As in the polynomial ring case,  we set the Lyubeznik numbers as
$$\lambda_{p,i}(R/I):=\mu^p(\fM, H_{I}^{n-i}(R)). $$

Work of the second author in  \cite{Yan01-2} states that this set of
invariants are well defined in this framework. Namely,
Theorem~\ref{Yan01, 3.10} holds verbatim in this situation.

\begin{theorem}[{\cite[Corollary~5.12]{Yan01-2}}]\label{Yan01-2, 5.12}
Let $R=\kk[C]$ be a normal simplicial semigroup ring which is Gorenstein,  and $I_\Delta$ a
radical monomial ideal.  Then we have
$$\lambda_{p,i}(R/I_\Delta) = \dim_\kk [\Ext_R^{n-p}(\Ext_R^{n-i}(R/I_\Delta, \omega_R), \omega_R)]_0< \infty.$$
\end{theorem}

Notice that in this setting we have that  whenever we have a
multigraded isomorphism $\kk[C]/I_\Delta \cong \kk[C']/I_{\Delta'}$
between quotients of Gorenstein normal simplicial semigroup rings by
radical monomial ideals, then the corresponding Lyubeznik numbers
coincide. This multigraded framework slightly differs from the
original situation for regular local rings stated in \cite{Ly93}.
However, as stated in \cite[Remark~5.14]{Yan01-2}, if $\Delta \cong \Delta'$ as simplicial complexes, then
$R/I_\Delta$ and $R'/I_{\Delta'}$ have the same Lyubeznik numbers.
In this sense, to study the Lyubeznik numbers of a quotient $R/I_\Delta$ of  a Gorenstein normal simplicial semigroup ring $R$
by a radical monomial ideal $I_\Delta$, we may assume that  $R$ is a polynomial ring and $R/I_\Delta$ is a Stanley-Reisner ring.
In Theorem~\ref{topological}, we will prove a stronger result.

\vskip 2mm

It is also worth to point out that several features of Lyubeznik
numbers are still true in this setting. 
In what follows, we assume that $I$ is a monomial ideal of $R$. 

\vskip 2mm

(1)   As in the polynomial ring case, we have the Euler characteristic equation,
$$\sum_{0\leq p,i \leq d} (-1)^{p-i} \la_{p,i}(R/I)=1.$$ 
Moreover, the statements corresponding to Theorem~\ref{consecutiveness} (the consecutiveness of nontrivial lines) 
still holds. 
In fact, we may assume that $I$ is a radical ideal, and hence $I=I_\Delta$ for some simplicial complex $\Delta$ and
then  reduce to the case when  $R$ is a polynomial ring as in  \cite[Remark~5.14 (b)]{Yan01-2}.

\vskip 2mm

If we assume that $I=\sqrt{I}$, Proposition~\ref{seq CM} also holds in the present situation.
However, we cannot drop this assumption, since we have no idea whether the condition of
being sequentially Cohen-Macaulay is preserved after taking radicals.
What is known is that if $R/I$ is Cohen-Macaulay then so is $R/\sqrt{I}$ (see \cite[Theorem~6.1]{Yan08}).
Hence if $R/I$ is  Cohen-Macaulay then the Lyubeznik table of $R/I$ is trivial. 


\vskip 2mm

(2)  For a radical monomial ideal $I_\Delta$ with $\dim R/I_\Delta =d$, the highest Lyubeznik number
$$\lambda_{d,d}(R/I_\Delta)=\dim_\kk [\Ext^{n-d}_R(\Ext^{n-d}_R(R/I_\Delta, \omega_R), \omega_R)]_0$$
has a simple topological (or combinatorial) meaning.  In fact, to study this number we may assume that $R$ is a polynomial ring,  and we can use
a combinatorial description of $$\Ext^{n-d}_R(\Ext^{n-d}_R(R/I_\Delta, \omega_R), \omega_R)$$
given in \cite[P.96]{Sta96}. Roughly speaking, $\lambda_{d,d}(R/I_\Delta)$ is the number of
``connected in codimension one components'' of $|\Delta|$.
(This result holds in a much wider context, see \cite{Zha07}.)
In particular, if $R/I_\Delta$ satisfies Serre's condition $(S_2)$
then $\lambda_{d,d}(R/I_\Delta)=1$, while the converse is not true.


\subsection{Lyubeznik table is a topological invariant}

Recall that if $R=\kk[C]$ is simplicial then an order ideal $\Delta$
of $L$ is essentially a simplicial complex, and hence it has the
geometric realization $|\Delta|$. It is natural to ask how Lyubeznik
numbers of $R/I_{\Delta}$ depend on $|\Delta|$. The next theorem
shows that Lyubeznik numbers are not only an algebraic invariant but
also a topological invariant.

\begin{theorem}\label{topological}
Let $R=\kk[C]$ be a simplicial normal semigroup ring which is
Gorenstein and $I_\Delta \subset R$  a radical monomial ideal. Then,
$\lambda_{p,i}(R/I_\Delta)$ depends only on the homeomorphism class
of $|\Delta|$ and $\operatorname{char}(\kk)$.
\end{theorem}

Bearing in mind Theorem~\ref{Yan01-2, 5.12}, it suffices to show
that $$\dim_\kk [\Ext_R^{n-p}(\Ext_R^{n-i}(R/I_\Delta, \omega_R),
\omega_R)]_0 $$  depends only on the topology of $|\Delta|$ and
$\operatorname{char}(\kk)$. For this statement, the assumption that
$R$ is simplicial and Gorenstein is irrelevant
 (if $R$ is not simplicial, then $\Delta$ is essentially a CW complex).
In \cite[Theorem~2.10]{OY07}, R. Okazaki and the second author showed that the invariant which is (essentially) equal to
$$\depth_R (\Ext_R^{n-i}(R/I_\Delta, \omega_R)) = \min \{ \, j \mid  \Ext^{n-j}_R(\Ext_R^{n-i}(R/I_\Delta, \omega_R), \omega_R ) \ne 0 \, \}$$
depends only on $|\Delta|$ and $\operatorname{char}(\kk)$ for each
$i$. Our proof here uses similar arguments to the aforementioned
result. To do so, we have to recall some previous work of the second
author in \cite{Yan03}.

\vskip 2mm

Recall that $P = \bR_{\geq 0} C$ is a polyhedral cone associated with the semigroup ring $R=\kk[C]$. 
We have a hyperplane $H \subset \bR^n$ such that $B:= H \cap P$
is an $(n-1)$-polytope (an $(n-1)$-simplex, if $R$ is simplicial).
For $F \in L$, set $|F|$ to be the relative interior of the face $F \cap H$ of $B$.
We can regard  an order ideal $\Delta \subseteq L$ as a CW complex  (a simplicial complex, if $R$ is simplicial) whose
geometric realization is $|\Delta| :=\bigcup_{F \in \Delta} |F| \subseteq B$.

For $F \in L$, $$U_F := \bigcup_{F' \in L, \, F' \supset F} |F'|$$ is an open set of $B$.
Note that $\{ \, U_F \mid \{ 0 \} \ne F \in L \, \}$
is an open covering of $B$.
In \cite{Yan03}, from $M \in \Sq R$, we constructed a sheaf $M^+$ on $B$.
More precisely, the assignment
$$\Gamma(U_F, M^+) = M_{\bc(F)}$$ for each $F \ne \{ 0  \}$ and
the map
$$\varphi_{F,G}^M:\Gamma(U_{G},M^+) = M_{\bc(G)} \longrightarrow M_{\bc(F)} =
\Gamma(U_F,M^+)$$ for $F, G \ne \{  0 \}$ with
$F \supset G$ (equivalently, $U_G \supset U_{F}$) defines
a sheaf. Note that $M_0$ is ``irrelevant'' to $M^+$.

\vskip 2mm

For example,  $(R/I_\Delta)^+ \cong j_* \const_{|\Delta|}$,
where $\const_{|\Delta|}$ is the constant sheaf on $|\Delta|$
with coefficients in $\kk$, and $j$ is the embedding map $|\Delta| \hookrightarrow B$.
Similarly, we have that $(\omega_R)^+ \cong h_! \const_{B^\circ}$,
where $\const_{B^\circ}$ is the constant sheaf on the relative interior $B^\circ$ of $B$,
and $h$ is the embedding map $B^\circ \hookrightarrow B$.
Note that $(\omega_R)^+$ is the orientation sheaf of $B$ with coefficients in $\kk$.

\vskip 2mm

Let $\Delta \subseteq L$ be an order ideal, and set $X := |\Delta| \subseteq B$.
For $M \in \Sq R$, $M$ is an
$R/I_\Delta$-modules  (i.e., $\operatorname{ann}(M) \supset I_\Delta$) if and only if $\Supp (M^+) := \{ x \in B \mid (M^+)_x \ne 0 \} \subseteq X$.
In this case,  we have
$$H^i(B;M^+)  \cong H^i(X;M^+|_X)$$
for all $i$. Here $M^+|_X$ is the restriction of the sheaf $M^+$ to the closed set $X \subseteq B$.
Combining this fact with \cite[Theorem~3.3]{Yan03}, we have the following.

\begin{theorem}[{c.f. \cite[Theorem~3.3]{Yan03}}]\label{Hoch}
With the above situation, we have
$$H^i(X; M^+|_X) \cong [H_\fM^{i+1}(M)]_0 \quad  \text{for all $i \geq 1$},$$
and an exact sequence
\begin{equation}\label{lower dimension}
0 \longrightarrow [H_\fM^0(M)]_0 \longrightarrow M_0 \longrightarrow  H^0( X; M^+|_X) \longrightarrow [H_\fM^1(M)]_0 \longrightarrow 0.
\end{equation}
In particular,
$[H_\fM^{i+1}(R/I_\Delta)]_0 \cong \widetilde{H}^i(X ; \kk)$ for all $i \geq 0$,
where $\widetilde{H}^i(X ; \kk)$ denotes the $i$th reduced cohomology of $X$ with coefficients in $\kk$.
\end{theorem}

Recall that  $X$ admits Verdier's dualizing complex $\Dcom_X$ with coefficients in $\kk$.
For example, $\Dcom_B$ is quasi-isomorphic to $(\omega_R)^+[n-1]$.
The former half of (1) of the next theorem is a restatement of \cite[Theorem~4.2]{Yan03}, and the rest is that of 
 \cite[Lemma~5.11]{Yan08}. 

\begin{theorem}[{\cite[Theorem~4.2]{Yan03} and \cite[Lemma~5.11]{Yan08}}]\label{Verdier}
With the above notation, we have the following:

\vskip 2mm

(1) $\Supp (\Ext_R^{n-i}(M, \omega_R)^+) \subseteq X$
and $$\Ext_R^{n-i}(M, \omega_R)^+ |_X \cong \cExt^{1-i}(M^+|_X, \Dcom_X).$$
Moreover, for $i \ge 2$, we have
$$[\Ext_R^{n-i}(M, \omega_R)^+]_0 \cong \Ext^{1-i}(M^+|_X, \Dcom_X).$$

(2) Via the isomorphisms in (1), for $i \ge 2$,
the natural map $$\Ext^{1-i}(M^+|_X, \Dcom_X)\longrightarrow   \Gamma(X; \cExt^{1-i}(M^+|_X, \Dcom_X))$$
coincides with the middle map
$$[\Ext^{n-i}_R(M, \omega_R)]_0 \longrightarrow \Gamma(X; \Ext^{n-i}_R(M, \omega_R)^+|_X)$$
of the sequence \eqref{lower dimension} for $\Ext^{n-i}_R(M, \omega_R) \in \Sq R$.
\end{theorem}

\noindent{\it The proof of Theorem~\ref{topological}.}
We show that the dimension of $[\Ext_R^{n-p}(\Ext_R^{n-i}(R/I_\Delta, \omega_R),   \omega_R)]_0$
$(\cong [H_\fM^p(\Ext_R^{n-i}(R/I_\Delta, \omega_R))^*]_0)$ depends only on $X$ and $\operatorname{char}(\kk)$.
If $p \ge 2$, then 
we have
$$[H_\fM^p(\Ext_R^{n-i}(R/I_\Delta, \omega_R))]_0 \cong H^{p-1}(X; \cExt^{1-i}(\const_X, \Dcom_X))$$
by Theorems~\ref{Hoch} and \ref{Verdier} (1). The right side of the equation clearly depends only on $X$ and $\operatorname{char}(\kk)$
for each $p,i$. Next we consider the case $p=0,1$. By Theorem~\ref{Hoch}, $H_\fM^0(\Ext_R^{n-i}(R/I_\Delta, \omega_R))$ and
$H_\fM^1(\Ext_R^{n-i}(R/I_\Delta, \omega_R))$ are the kernel and the cokernel of the map
$$[\Ext^{n-i}_R(R/I_\Delta, \omega_R)]_0 \longrightarrow \Gamma(X; \Ext^{n-i}_R(R/I_\Delta, \omega_R)^+|_X)$$
respectively.  If $i \ge 2$, the above map is equivalent to the natural map
$$\Ext^{1-i}(\const_X, \Dcom_X) \longrightarrow  \Gamma(X; \cExt^{1-i}(\const_X, \Dcom_X))$$
by Theorem~\ref{Verdier} (2), and the dimensions of its kernel and cokernel are invariants of $X$.

It remains to show the case ($p=0,1$ and) $i=0,1$.
Clearly, $\Ext^n_R(R/I_\Delta, \omega_R) \ne 0$,
if and only if  $\Ext^n_R(R/I_\Delta, \omega_R) = \kk$,
if and only if $I_\Delta = \fM$, if and only if $X=\emptyset$.
Hence $\lambda_{0,0}(R/I_\Delta) \ne 0$, if and only if  $\lambda_{0,0}(R/I_\Delta) = 1$, if and only if $X=\emptyset$.
On the other hand, it is easy to check out that
$\lambda_{1,1}(R/I_\Delta)$ is always ``trivial'', that is,
$$\lambda_{1,1}(R/I_\Delta)=
\begin{cases}
1 & \text{if $\dim (R/I_\Delta)=1$ \ (i.e., $\dim |\Delta| = 0$),} \\
0 & \text{otherwise}
\end{cases}
$$
(the same is true for the local ring case using the spectral sequence argument as in the proof of Theorem~\ref{consecutiveness}
or adapting the techniques used in \cite{Wa2}).
Hence the remaining case is only $\lambda_{0,1}(R/I_\Delta)$, but the following fact holds. 

\medskip

\noindent{\bf Claim.} {\it 
If $R=\kk[C]$ is a simplicial normal semigroup ring which is Gorenstein, then we have
$$
\lambda_{0,1}(R/I_\Delta)=
\begin{cases}
c-1 & \text{if $\dim (R/I_\Delta) \ge 2$  \ (i.e., $\dim |\Delta| \ge 1$),}\\
0 & \text{otherwise,}
\end{cases}
$$
where $c$ is the number of the connected components of $|\Delta'|:=| \Delta| \setminus \{ \text{isolated points}\}$.
}

\medskip

Let us prove the claim.  
We may assume that $\dim (R/I_\Delta) > 0$.
If $\dim (R/I_\Delta)=1$, then $R/I_\Delta$ is Cohen-Macaulay, and the  assertion is clear.
So we may assume that $\dim (R/I_\Delta) \ge 2$.
First, we consider the case when $I_\Delta$ does not have 1-dimensional associated
primes, equivalently, $|\Delta|$ does not admit isolated points (i.e., $|\Delta|=|\Delta'|$).
Then  we have $$\dim_R (\Ext_R^{n-1}(R/I_\Delta, \omega_R)) <1.$$
Since $\Ext_R^{n-1}(R/I_\Delta, \omega_R)$ is a squarefree module, we have
$$\Ext_R^{n-1}(R/I_\Delta, \omega_R)=[\Ext_R^{n-1}(R/I_\Delta, \omega_R)]_0.$$
We also have
$$[\Ext_R^{n-1}(R/I_\Delta, \omega_R)]_0 \cong [H_\fM^1(R/I_\Delta)]_0 \cong \widetilde{H}^0(X ; \kk) \cong \kk^{c-1},$$
where the second isomorphism follows from the last statement of Theorem~\ref{Hoch}.
Hence
$$\lambda_{0,1}(R/I_\Delta)= \dim_\kk [\Ext_R^n(\Ext_R^{n-1}(R/I_\Delta, \omega_R), \omega_R)]_0 = \dim_\kk [\Ext_R^n(\kk^{c-1}, \omega_R)]_0 =c-1,$$
and we are done.

\vskip 2mm

So we now consider the case where $I_\Delta$ admits 1-dimensional associated primes. Set $I:=I_{\Delta'}$. Then there is a monomial ideal
$J$ of $R$ with $I_\Delta = I \cap J$ and  $\dim R/J=1$.
Note that $I+J = \fM$.
The short exact sequence $0 \to R/I_\Delta \to R/I \oplus R/J \to R/\fM \, (\cong \kk) \to 0$ yields the exact sequence
\begin{equation}\label{ext sequence}
0 \longrightarrow \Ext^{n-1}_R(R/I, \omega_R) \oplus \Ext^{n-1}_R(R/J, \omega_R) \longrightarrow \Ext^{n-1}_R(R/I_\Delta, \omega_R)
\longrightarrow  \kk \longrightarrow 0.
\end{equation}
Since Lyubeznik numbers of type $\lambda_{1,1}(-)$ are always
trivial,  we have
$$[\Ext^{n-1}_R(\Ext^{n-1}_R(R/I_\Delta, \omega_R), \omega_R)]_0 = [\Ext^{n-1}_R(\Ext^{n-1}_R(R/I, \omega_R), \omega_R)]_0 =0$$
and $[\Ext^{n-1}_R(\Ext^{n-1}_R(R/J, \omega_R), \omega_R)]_0 = \kk$.
It is also clear that $\Ext^n_R(\Ext^{n-1}_R(R/J, \omega_R),
\omega_R) =0.$ Thus applying $\Ext_R^\bullet(-, \omega_R)$ to
\eqref{ext sequence}, we obtain
\begin{eqnarray*}
0 &\longrightarrow& [\Ext^{n-1}_R(\Ext^{n-1}_R(R/J, \omega_R), \omega_R)]_0 \  (\cong \kk) \longrightarrow
[\Ext^n_R(\kk, \omega_R)]_0 \ (\cong \kk) \longrightarrow\\
&&  [\Ext^n_R(\Ext^{n-1}_R(R/I_\Delta, \omega_R), \omega_R)]_0 \longrightarrow [\Ext^n_R(\Ext^{n-1}_R(R/I, \omega_R), \omega_R)]_0
\longrightarrow 0 .
\end{eqnarray*}
Since $[\Ext^n_R(\Ext^{n-1}_R(R/I, \omega_R), \omega_R)]_0 \cong \kk^{c-1}$ as we have shown above,
it follows that $$[\Ext^n_R(\Ext^{n-1}_R(R/I_\Delta, \omega_R), \omega_R)]_0 \cong \kk^{c-1},$$
and we are done.
\qed 

\begin{example}
This example concerns the final step of the proof of Theorem~\ref{topological}. 
Let $R=\kk[x_1, \ldots, x_7]$ be a polynomial ring, and consider the monomial ideal
$$I_\Delta=(x_2,x_3,x_4,x_5,x_6, x_7) \cap (x_1,x_4,x_5, x_6,x_7) \cap (x_1,x_2,x_3,x_6, x_7) \cap (x_1, x_2,x_3,x_4,x_5).$$
Then $|\Delta|$ consists of 1 isolated point and 3 segments, see Fig 1 below.
So $|\Delta'|$,  which is $|\Delta| \setminus \{ v_1 \}$, consists of 3 segments.
We have $\lambda_{0,1}(R/I_\Delta)=3-1=2$.
\begin{figure}[htbp]
\begin{center}
\includegraphics[height=4cm]{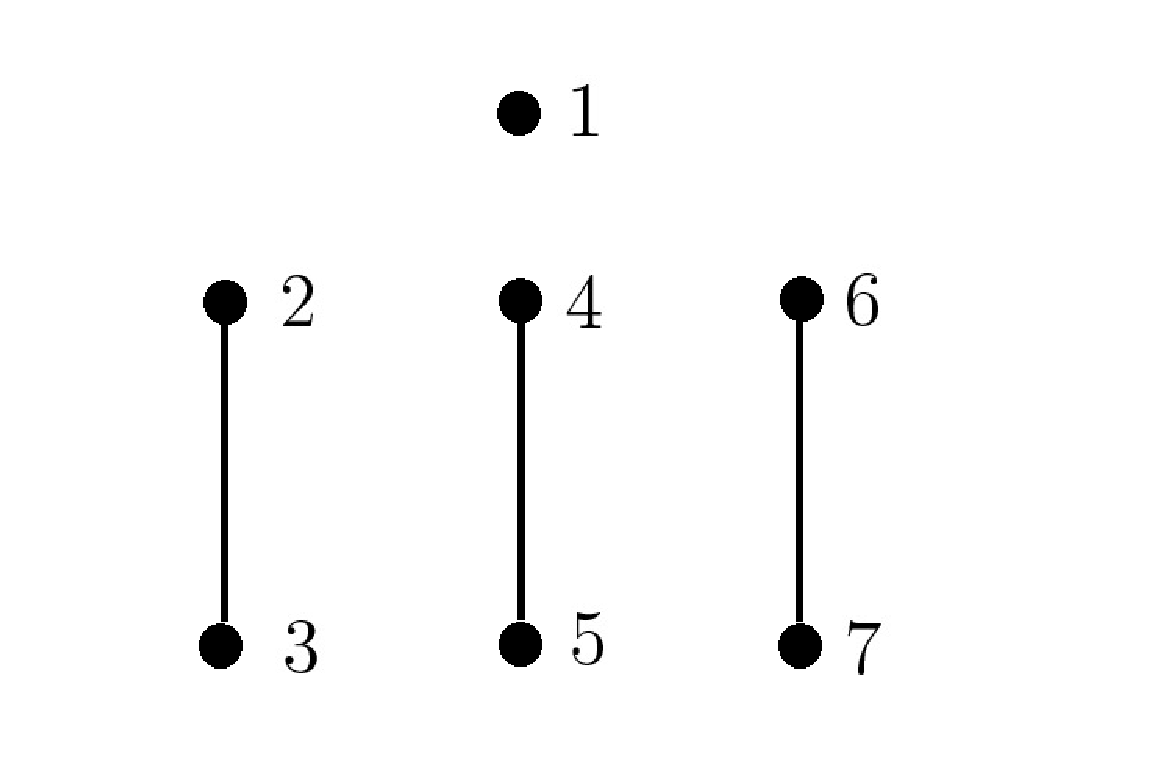}
\end{center}
\caption{}
\end{figure}
\end{example}

\end{document}